\documentclass[10pt,a4paper]{article}
\usepackage{amsthm,amsmath,amssymb,amsfonts,hyperref}
\newtheorem{theorem}{Theorem}

\numberwithin{equation}{section}
\numberwithin{lemma}{section}
\numberwithin{theorem}{section}
\numberwithin{corollary}{section}

\allowdisplaybreaks

\begin{document}
\setcounter{page}{1}

\title{On the matrix version of new extended Gauss, Appell and Lauricella hypergeometric  functions}

\author{Ashish Verma \footnote{Department of Mathematics, Prof. Rajendra Singh (Rajju Bhaiya), Institute of Physical Sciences for Study and Research   V.B.S. Purvanchal University, Jaunpur  (U.P.)- 222003, India. \newline Email: vashish.lu@gmail.com}, \, Ravi Dwivedi\footnote {Department of Mathematics, National Institute of Technology, Kurukshetra, India. \newline E-mail: dwivedir999@gmail.com (Corresponding author)}} 

\maketitle
\begin{abstract}
Inspired by certain interesting recent extension of the gamma, beta and hypergeometric matrix  functions, we introduce here  new   extension of the gamma  and beta matrix function. We also introduce new extensions of the  Gauss hypergeometric matrix function, confluent hypergeometric matrix function,  Appell matrix function and Lauricella matrix function of three variables in terms of the new extended beta matrix function.  Then we investigate certain properties of these extended  matrix functions such as the integral  representations, differential formulae and recurrence relations.

\medskip
\noindent\textbf{Keywords}: Matrix functional calculus, Gauss hypergeometric function,  Appell functions, Lauricella hypergeometric functions.

\medskip
\noindent\textbf{AMS Subject Classification}: 15A15; 33C65. 
\end{abstract}

\section{Introduction}
Theory of special functions is being initiated with the study of gamma and beta functions. Further, the Gauss hypergeometric function and Kummer hypergeometric function played crucial role in the study of special functions. These initial special functions received much attention by mathematician as well as physicist due to various application of functions in the field of mathematics, physics, engineering and Lie theory. This motivate the study of the extension of these functions. In last few years, several extensions of gamma function, beta function and Gauss hypergeometric function have been considered, See \cite{ao, cqrs,  cqs, C, n1, rp}. Earlier, a new extension of gamma and beta functions  in terms of confluent hypergeometric function have been studied and using  new generalization of  beta function,  Gauss hypergeometric function, Appell functions and Lauricella functions of three variables have been defined  \cite{liu,ema}. 

The theory of special matrix function has been initiated by J\'odar and cort\'es, \cite{jjc98a, jjc98b} and this work has been carried for several variable special functions in \cite{MA}, \cite{QM}, \cite{RD}-\cite{RD3}. Furthermore,  the extension of the gamma, beta and hypergeometric matrix functions have been given in \cite{ab1}, \cite{ab11},\cite{ab2}, \cite{ZM}, \cite{vds}. Also Bakhet {\em et al}. \cite{b}  have studied the matrix version of extended Bessel function and discussed the integral representations, differentatial formula, hypergeometric representation of such functions. In this paper, we studied the matrix version of new extended hypergeometric function, more explicitly new extension of gamma matrix function, beta matrix function, Gauss hypergeometric matrix function, Appell matrix functions and Lauricella matrix function of three variables has been introduced. We also give integral representation, differential formulae and recurrence formulae for these new extended matrix functions. The section-wise treatment is as follows:     
 
In Section~2, we give the basic definitions related to special matrix functions that are needed in the sequel.  In Section~3, we define the new extended gamma and beta matrix functions. In Section~4, we define the matrix analogue of new extended Gauss hypergeometric function and Kummer hypergeometric function. We also give integral representations, transformations and differential formulae satisfied by them. Finally, in Section~5, we introduce new extended Appell matrix functions and extended Lauricella matrix function of three variables. The integral representations, differential formulae and recursion formulae for these new extended matrix functions are also determined.

\section{Preliminaries}
Throughout the paper, $\mathbb{C}^{r\times r}$ is the vector space of $r$-square matrices with complex entries. For a matrix   $A\in \mathbb{C}^{r\times r}$ the spectrum, denoted by $\sigma(A)$, is the set of eigenvalues of $A$. 

If $\Re(z)$ denotes the real part of a complex number $z$, then a matrix $A$ in $\mathbb{C}^{r\times r}$  is said to be positive stable if $\Re(\lambda)>0$ for all $\lambda\in\sigma(A)$. 
 
 If $A$ is a positive stable matrix in $\mathbb{C}^{r \times r}$, then $\Gamma(A)$ can be expressed as \cite{jjc98a}
 \begin{equation}
 \Gamma(A) = \int_{0}^{\infty} e^{-t} \, t^{A-I}\, dt.\label{1a1.4}
 \end{equation} 
Furthermore, if $A+kI$ is invertible for all integers $k\geq 0$, then the reciprocal gamma matrix function is defined as \cite{jjc98a}
\begin{equation}
\Gamma^{-1}(A)= A(A+I)\dots (A+(n-1)I)\Gamma^{-1}(A+nI) , \  n\geq 1.\label{eq.07}
\end{equation}
By  application of the matrix functional calculus, the Pochhammer symbol  for  $A\in \mathbb{C}^{r\times r}$ is given by \cite{jjc98b}
\begin{equation}
(A)_n = \begin{cases}
I, & \text{if $n = 0$},\\
A(A+I) \dots (A+(n-1)I), & \text{if $n\geq 1$}.
\end{cases}\label{c1eq.09}
\end{equation}
This gives
\begin{equation}
(A)_n = \Gamma^{-1}(A) \ \Gamma (A+nI), \qquad n\geq 1.\label{c1eq.010}
\end{equation} 
 If $A$ and $B$ are positive stable matrices in $\mathbb{C}^{r \times r}$, then, for $AB = BA$, the beta matrix function is defined as \cite{jjc98a}
\begin{align}
\mathfrak{B}(A,B) =\Gamma(A)\Gamma(B)\Gamma^{-1}(A+B) &= \int_{0}^{1}t^{A-I}(1-t)^{B-I}dt \label{1ca1.4}\\
&=\int_{0}^{\infty}u^{A-I}(1+u)^{-(A+B)}du.\label{1ca1.5}
\end{align}
For positive stable matrices $\mathbb{X}$, $A$, the extension of gamma matrix function given by, \cite{ab1}
\begin{align}
	\Gamma_{\mathbb{X}}(A) = \int_{0}^{\infty} t^{A-I} \exp (-tI-\mathbb{X} t^{-1}).
\end{align}
Let $A$, $B$ and $\mathbb{X}$ be positive stable  and commuting matrices in $\mathbb{C}^{r\times r}$ such that $A+kI$,  $B+kI$ and $\mathbb{X}+kI$ are invertible for all integer $k\geq 0$. Then the extended beta matrix function $\mathfrak{B}(A, B; \mathbb{X})$ is defined by \cite{ab1}
\begin{align}
\mathfrak{B}(A, B; \mathbb{X})=\int_{0}^{1} t^{A-I} (1-t)^{B-I} \exp\Big(\frac{-\mathbb{X}}{t(1-t)}\Big)dt.\label{xb1}
\end{align}
Hence,
\begin{align}
\mathfrak{B}(A, B; \mathbb{X})= \Gamma({A}, \mathbb{X}) \,\Gamma({B}, \mathbb{X})\, \Gamma^{-1}({A+B}, \mathbb{X}).
\end{align}
It is obvious that $\mathbb{X}=O$ gives the original beta matrix  function \cite{jjc98a}.
Let $A$, $B$, $C$ be positive stable matrices in $\mathbb{C}^{r\times r}$ such that $C+kI$ is invertible for all integers $k\ge 0$. Then the Gauss hypergeometric matrix function is defined by \cite{jjc98b}
\begin{align}
{}_2F_1 (A, B; C; z) = \sum_{n=0}^{\infty} (A)_n (B)_n (C)_n^{-1} \frac{z^n}{n!}.\label{52.9}
\end{align}
The series \eqref{52.9} converges absolutely for $\vert z\vert < 1$ and for $z = 1$, if $\alpha(A) + \alpha(B) < \beta(C)$, where $\alpha(A) = \max\{\, \Re(z) \mid z\in \sigma(A)\, \}$, $\beta(A) = \min\{\, \Re(z) \mid z \in \sigma(A)\, \}$ and $\beta(A) = -\alpha(-A)$.

 Let $A$, $B$, $C$, $C-B$ and $\mathbb{X}$ be positive stable matrices in $\mathbb{C}^{r \times r}$ such that $CB = BC$, $C\mathbb{X} = \mathbb{X}C$ and $B \mathbb{X} = \mathbb{X} B$. Then the extended Gauss hypergeometric matrix function (EGHMF) and  extended Kummer hypergeometric matrix function (EKHMF) are defined by \cite{ab2}
\begin{align}
&F^{(\mathbb{X})}(A, B; C; z)\notag\\
&=  \left(\sum_{m\geq 0} (A)_{m} \, \mathfrak{B}(B+mI, C-B; \mathbb{X})\frac{z^{m}}{m!}\right) \times \Gamma(C)\Gamma^{-1}(B)\Gamma^{-1}(C-B);\label{eg1}
\end{align}
and
\begin{align}
&\phi^{(\mathbb{X})}(B; C; z)\notag\\
&=  \left(\sum_{m\geq 0} \, \mathfrak{B}(B+mI, C-B; \mathbb{X})\frac{z^{m}}{m!}\right) \times \Gamma(C)\Gamma^{-1}(B)\Gamma^{-1}(C-B)\label{kh1}
\end{align}
respectively.

\section{The new extended  gamma and beta matrix  functions}
In this section, we consider the generalization of gamma and beta matrix functions. For positive stable matrices $A$, $B$, $\mathbb{X}$ and $\mathbb{Y}$, we define the new extension of gamma matrix function as 
\begin{align}
	\Gamma_{\mathbb{Y}}^{(A, B)} (\mathbb{X}) = \int_{0}^{\infty} {}_1F_1 (A; B;-tI-\mathbb{Y} t^{-1}) \, t^{\mathbb{X}-I} .\label{3.1}
\end{align}
Let $A$, $B$, $\mathbb{X}$, $\mathbb{Y}$ and $\mathbb{Z}$ be positive stable matrices in $\mathbb{C}^{r\times r}$. Then, we introduce the new extension of beta matrix function, denoted by $\mathfrak{B}_{\mathbb{Y}}^{(A, B)} (\mathbb{X}, \mathbb{Z})$
\begin{align}
	\mathfrak{B}_{\mathbb{Y}}^{(A, B)} (\mathbb{X}, \mathbb{Z}) = \int_{0}^{1}  {}_1F_1 \left(A; B; -\frac{\mathbb{Y}}{t(1-t)}\right) \, t^{\mathbb{X}-I} (1-t)^{\mathbb{Z}-I}  dt.\label{3.2}
\end{align}
From Equations \eqref{3.1} and \eqref{3.2}, the particular case is given by $\Gamma_{\mathbb{Y}}^{(A, A)} (\mathbb{X}) = \Gamma_{\mathbb{Y}} (\mathbb{X})$, $\Gamma_{\bf{0}}^{(A, A)} (\mathbb{X}) = \Gamma (\mathbb{X})$, $\mathfrak{B}_{\mathbb{Y}}^{(A, A)} (\mathbb{X}, \mathbb{Z}) = \mathfrak{B}_{\mathbb{Y}} (\mathbb{X}, \mathbb{Z})$ and  $\mathfrak{B}_{\bf{0}}^{(A, A)} (\mathbb{X}, \mathbb{Z}) = \mathfrak{B} (\mathbb{X}, \mathbb{Z})$.

We now give some results related to new extended gamma and beta matrix functions.
\begin{theorem}
Let $A$, $B$, $A-B$, $\mathbb{X}$ and $\mathbb{Y}$ be positive stable matrices in $\mathbb{C}^{r\times r}$ such that $A$, $B$, $\mathbb{X}$ and $\mathbb{Y}$ are commuting. Then the new extended gamma matrix function can also be presented in the integral form as
\begin{align}
\Gamma_\mathbb{\mathbb{Y}}^{(A, B)} (\mathbb{X}) = \ \Gamma(B) \Gamma^{-1} (A) \Gamma^{-1} (B-A) \int_{0}^{1} \Gamma_{\mathbb{Y} {\mu}^2} (\mathbb{X})  \ \mu^{A-\mathbb{X} - I}  \ (1-\mu)^{B - A - I} d\mu .\label{3.3}
\end{align}
\end{theorem}
\begin{proof}
	Using the integral representation of confluent hypergeometric matrix function \cite{cdss}, we have
\begin{align}
	\Gamma_{\mathbb{Y}}^{(A, B)} (\mathbb{X}) =  \Gamma(B) \Gamma^{-1} (A) \Gamma^{-1} (B-A) \int_{0}^{\infty} \int_{0}^{1} u^{\mathbb{X} - I} e^{-ut-\frac{\mathbb{Y} t}{u}}  \ t^{A- I}  \ (1-t)^{B - A - I} dt du.
\end{align}
Consider the following transformation $v = ut$, $\mu = t$ and the Jacobian of the transformation  $J = \frac{1}{\mu}$, we have
 \begin{align}
 	\Gamma_{\mathbb{Y}}^{(A, B)} (\mathbb{X}) =\Gamma(B) \Gamma^{-1} (A) \Gamma^{-1} (B-A) \int_{0}^{\infty} \int_{0}^{1} v^{\mathbb{X} - I} e^{-v-\frac{\mathbb{Y} \mu^2}{v}}  \, dv \ \mu^{A - \mathbb{X} - I}  \ (1-\mu)^{B - A - I} d\mu.\label{3.5}
 \end{align}
The convergence of matrix functions permit us to interchange the order of integration, Equation \eqref{3.5} yields
 \begin{align}
	\Gamma_{\mathbb{Y}}^{(A, B)} (\mathbb{X}) &=   \Gamma(B) \Gamma^{-1} (A) \Gamma^{-1} (B-A) \int_{0}^{1} \left[\int_{0}^{\infty} v^{\mathbb{X} - I} e^{-v-\frac{\mathbb{Y} \mu^2}{v}}  \, dv\right] \ \mu^{A - \mathbb{X} - I}  \ (1-\mu)^{B - A - I} d\mu \nonumber\\
	& = \Gamma(B) \Gamma^{-1} (A) \Gamma^{-1} (B-A) \int_{0}^{1} \Gamma_{\mathbb{Y} {\mu}^2} (\mathbb{X})  \ \mu^{A-\mathbb{X} - I}  \ (1-\mu)^{B - A - I} d\mu.
\end{align}
It completes the proof of \eqref{3.3}.
\end{proof}
\begin{theorem}
For the positive stable matrices $A$, $B$, $A+B$, $\mathbb{X}$, $\mathbb{Y}$ and $\mathbb{Z}$, we have another integral representation of new extended beta matrix function
\begin{align}
\mathfrak{B}_{\mathbb{Y}}^{(A, B)} (\mathbb{X}, \mathbb{Z}) = \int_{0}^{\infty} {}_1F_1 \left(A; B; -2 \mathbb{Y} - \mathbb{Y}\left(u + \frac{1}{u}\right)\right)  u^{\mathbb{X}-I} (1+u)^{-(\mathbb{X}+\mathbb{Z})}  dt. \label{e3.7}
\end{align}
\end{theorem}
\begin{proof}
Letting $t = \frac{u}{1+u}$ and using the definition of new extended beta matrix function, we get the required result \eqref{e3.7}.
\end{proof}
\begin{theorem}
	For the positive stable matrices $A$, $B$, $\mathbb{X}$, $\mathbb{Y}$ and $\mathbb{Z}$, the new extended beta matrix function satisfies the following relation
	\begin{align}
	\mathfrak{B}_{\mathbb{Y}}^{(A, B)} (\mathbb{X} + I, \mathbb{Z}) + \mathfrak{B}_{\mathbb{Y}}^{(A, B)} (\mathbb{X}, \mathbb{Z} + I) = \mathfrak{B}_{\mathbb{Y}}^{(A, B)} (\mathbb{X}, \mathbb{Z}). \label{3.7}
	\end{align}
\begin{proof}
From the definition of new extended beta matrix function, we have
\begin{align}
&\mathfrak{B}_{\mathbb{Y}}^{(A, B)} (\mathbb{X} + I, \mathbb{Z}) + \mathfrak{B}_{\mathbb{Y}}^{(A, B)} (\mathbb{X}, \mathbb{Z} + I)\nonumber\\
& = \int_{0}^{1} {}_1F_1 \left(A; B; -\frac{\mathbb{Y}}{t(1-t)}\right) t^{\mathbb{X}} (1-t)^{\mathbb{Z}-I}  dt \nonumber\\
& \quad + \int_{0}^{1} {}_1F_1 \left(A; B; -\frac{\mathbb{Y}}{t(1-t)}\right) t^{\mathbb{X}-I} (1-t)^{\mathbb{Z}}  dt\nonumber\\
& = \int_{0}^{1}  {}_1F_1 \left(A; B; -\frac{\mathbb{Y}}{t(1-t)}\right) [t^{\mathbb{X}} (1-t)^{\mathbb{Z}-I} +   t^{\mathbb{X}-I} (1-t)^{\mathbb{Z}}] dt\nonumber\\
& = \int_{0}^{1} {}_1F_1 \left(A; B; -\frac{\mathbb{Y}}{t(1-t)}\right) t^{\mathbb{X}-I} (1-t)^{\mathbb{Z}-I}  dt = \mathfrak{B}_{\mathbb{Y}}^{(A, B)} (\mathbb{X}, \mathbb{Z}).
\end{align}
\end{proof}
\end{theorem}
\begin{theorem}
	For positive stable matrices $A$, $B$, $\mathbb{X}$, $\mathbb{Y}$ and $I -\mathbb{Z}$, the new extended beta matrix function satisfies the following summation identity
	\begin{align}
\mathfrak{B}_{\mathbb{Y}}^{(A, B)} (\mathbb{X}, I - \mathbb{Z}) = \sum_{n = 0}^{\infty}  \ \mathfrak{B}_{\mathbb{Y}}^{(A, B)} (\mathbb{X} + nI, I) \, \frac{(\mathbb{Z})_n}{n!}. \label{3.10}  
	\end{align}
\end{theorem}
\begin{proof}
	From Equation \eqref{3.2}, we have
	\begin{align}
\mathfrak{B}_{\mathbb{Y}}^{(A, B)} (\mathbb{X}, I - \mathbb{Z})  = \int_{0}^{1}  {}_1F_1 \left(A; B; -\frac{\mathbb{Y}}{t(1-t)}\right) \, t^{\mathbb{X} - I} (1-t)^{-\mathbb{Z}}  dt. \label{3.11}
	\end{align}
Using the matrix identity $(1-t)^{-\mathbb{Z}} = \sum_{n = 0}^{\infty} \frac{(\mathbb{Z})_n}{n!}$, Equation \eqref{3.11} yields
	\begin{align}
	\mathfrak{B}_{\mathbb{Y}}^{(A, B)} (\mathbb{X}, I - \mathbb{Z})  & = \int_{0}^{1} \sum_{n = 0}^{\infty}  {}_1F_1 \left(A; B; -\frac{\mathbb{Y}}{t(1-t)}\right) \,   t^{\mathbb{X} + (n-1)I} \, \frac{(\mathbb{Z})_n}{n!} dt\nonumber\\
	& = \sum_{n = 0}^{\infty} \int_{0}^{1}  {}_1F_1 \left(A; B; -\frac{\mathbb{Y}}{t(1-t)}\right) \,    t^{\mathbb{X} + (n-1)I} \frac{(\mathbb{Z})_n}{n!}  dt\nonumber\\
	& = \sum_{n = 0}^{\infty}  \mathfrak{B}_{\mathbb{Y}}^{(A, B)} (\mathbb{X} + nI, I) \, \frac{(\mathbb{Z})_n}{n!}.
\end{align}
It completes the proof.
\end{proof}
\section{New extended Gauss hypergeometric and confluent hypergeometric matrix functions}
In this section, we define a new extension of Gauss hypergeometric  and confluent hypergeometric matrix function in terms of new extended beta matrix function. Several properties of these matrix function have also been studied. Let $A$, $B$, $A_1$, $B_1$, $C_1$, $C_1 - B_1$ and $\mathbb{Y}$ be matrices in $\mathbb{C}^{r \times r}$ such that $C_1+kI$ is invertible for all integers $k \ge 0$. Then, we define the new extended Gauss hypergeometric and confluent hypergeometric matrix functions as
\begin{align}
_2F_1^{(A, B; \mathbb{Y})} (A_1, B_1; C_1; z) = \sum_{n=0}^{\infty}  (A_1)_n \,	\mathfrak{B}_{\mathbb{Y}}^{(A, B)} (B_1 + nI, C_1-B_1) \,  [\mathfrak{B} (B_1, C_1-B_1)]^{-1}  \, \frac{z^n}{n!}\label{4.1}	
\end{align}
and
\begin{align}
	_1F_1^{(A, B; \mathbb{Y})} ( B_1; C_1; z) = \sum_{n=0}^{\infty} 	\mathfrak{B}_{\mathbb{Y}}^{(A, B)} (B_1 + nI, C_1-B_1) [\mathfrak{B} (B_1, C_1-B_1)]^{-1}  \, \frac{z^n}{n!}\label{4.2}	
\end{align}
respectively.

These new extensions of Gauss hypergeometric and confluent hypergeometric matrix function will be called as new extended Gauss hypergeometric matrix function (NEGHMF) and new extended confluent hypergeometric matrix function (NECHMF). We have following observations: $_2F_1^{(A, A; \mathbb{Y})} (A_1, B_1; C_1; z) = {} _2F_1^{\mathbb{Y})} (A_1, B_1; C_1; z)$, $_2F_1^{(A, A; 0)} (A_1, B_1; C_1; z) = {} _2F_1  (A_1, B_1; C_1; z)$ and  $_1F_1^{(A, A; \mathbb{Y})} ( B_1; C_1; z) = \phi^{\mathbb{Y}} ( B_1; C_1; z)$, $_1F_1^{(A, A; 0)} ( B_1; C_1; z) =  {}_1F_1 ( B_1; C_1; z)$. 
\begin{theorem}\label{t4.1}
For positive stable matrices  $A$, $B$, $A_1$, $B_1$, $C_1$, $C_1 - B_1$ and $\mathbb{Y}$ in $\mathbb{C}^{r\times r}$, the NEGHMF have following integral representations
\begin{align}
_2F_1^{(A, B; \mathbb{Y})} (A_1, B_1; C_1; z) & =  \int_{0}^{1} (1-zt)^{-A_1} {}_1F_1 \left(A; B; -\frac{\mathbb{Y}}{t(1-t)}\right)  t^{B_1-I}  \ (1-t)^{C_1-B_1-I}   dt \nonumber\\
& \quad \times [\mathfrak{B} (B_1, C_1-B_1)]^{-1},\label{4.3}
\end{align}
\begin{align}
	_2F_1^{(A, B; \mathbb{Y})} (A_1, B_1; C_1; z) & =  \int_{0}^{\infty} (1 + u(1-z))^{-A_1} \, (1+u)^{A_1} {}_1F_1 \left(A; B; -2 \mathbb{Y} - \mathbb{Y} \left(u + \frac{1}{u}\right)\right)   \nonumber\\
	& \quad \times  u^{B_1-I} \, (1+u)^{-C_1} du \ [\mathfrak{B} (B_1, C_1-B_1)]^{-1}.\label{a4.4}
\end{align}
\end{theorem} 
\begin{proof}
Using the integral representations of new extended beta matrix functions given in Equations \eqref{3.2} and \eqref{e3.7} and the matrix identity $(1-z)^{-A} = \sum_{n = 0}^{\infty} (A)_n \frac{z^n}{n!}$, the required results \eqref{4.3} and \eqref{a4.4} can be proved easily.
\end{proof}
\begin{theorem}\label{t4.2}
For positive stable matrices  $A$, $B$, $B_1$, $C_1$, $C_1 - B_1$ and $\mathbb{Y}$ in $\mathbb{C}^{r\times r}$, the NECHMF have following integral representations
\begin{align}
{}_1F_1^{(A, B; \mathbb{Y})} (B_1; C_1; z) & =  \int_{0}^{1} \exp(zt) \,{}_1F_1 \left(A; B; -\frac{\mathbb{Y}}{t(1-t)}\right) \, t^{B_1-I}  \ (1-t)^{C_1-B_1-I}   dt \nonumber\\
	& \quad \times [\mathfrak{B} (B_1, C_1-B_1)]^{-1},\label{4.5}
\end{align}
\begin{align}
	_1F_1^{(A, B; \mathbb{Y})} (B_1; C_1; z) & =  \int_{0}^{1} \exp (z(1-u)) \ {}_1F_1 \left(A; B;  \frac{-\mathbb{Y}}{u(1-u)} \right) \ u^{C_1-B_1-I}      \nonumber\\
	& \quad \times  (1-u)^{B_1-I} du \ [\mathfrak{B} (B_1, C_1-B_1)]^{-1}.\label{4.6}
\end{align}
\end{theorem}
The proof of Theorem~\ref{t4.2} is similar to Theorem~\ref{t4.1} and hence omitted.
\begin{theorem}
Let $A$, $B$, $\mathbb{Y}$, $A_1$, $B_1$, $C_1$ and  $C_1-B_1$ be positive stable matrices in $\mathbb{C}^{r \times r}$ such that $B_1 C_1 = C_1 B_1$. Then, we have the following differential formula satisfies by NEGHMF
	\begin{align}
		& \frac{d^n}{dz^n} \ {}_2F_1^{(A, B; \mathbb{Y})} (A_1, B_1; C_1; z)\nonumber\\
		& = (A_1)_n \ {}_2F_1^{(A, B; \mathbb{Y})} (A_1 +nI, B_1+nI; C_1+nI; z) \ (B_1)_n \ (C_1)_n^{-1}.\label{4.7}
	\end{align}
\end{theorem}
\begin{proof}
	From the definition of NEGHMF, we have 
	\begin{align}
	& \frac{ d }{ dz } \ {}_2F_1^{(A, B; \mathbb{Y})} (A_1, B_1; C_1; z)\nonumber\\
& = \frac{ d }{ dz } \sum_{n=0}^{\infty} (A_1)_n \,	\mathfrak{B}_{\mathbb{Y}}^{(A, B)} (B_1 + nI, C_1-B_1) [\mathfrak{B} (B_1, C_1-B_1)]^{-1}  \, \frac{z^n}{n!}\nonumber\\
& =  \sum_{n=1}^{\infty} (A_1)_n \,	\mathfrak{B}_{\mathbb{Y}}^{(A, B)} (B_1 + nI, C_1-B_1) [\mathfrak{B} (B_1, C_1-B_1)]^{-1}  \, \frac{z^{n-1}}{(n-1)!}\nonumber\\
& = \sum_{n=0}^{\infty} (A_1)_{n+1} \,	\mathfrak{B}_{\mathbb{Y}}^{(A, B)} (B_1 + (n+1)I, C_1-B_1) [\mathfrak{B} (B_1, C_1-B_1)]^{-1}  \, \frac{z^{n}}{n!}\nonumber\\
& = A_1 \, \sum_{n=0}^{\infty} (A_1 + I)_{n} \,	\mathfrak{B}_{\mathbb{Y}}^{(A, B)} (B_1 + (n+1)I, C_1-B_1) [\mathfrak{B} (B_1+I, C_1-B_1)]^{-1}  \nonumber\\
&  \quad \times  \frac{z^{n}}{n!}( B_1) (C_1)^{-1}\nonumber\\
& = (A_1)_1 \ {}_2F_1^{(A, B; \mathbb{Y})} (A_1 + I, B_1+I; C_1+I; z) \ (B_1)_1 \ (C_1)_1^{-1}.
	\end{align}
Continue this procedure $n$-times, we get the differential formula as
\begin{align}
	& \frac{d^n}{dz^n} \ {}_2F_1^{(A, B; \mathbb{Y})} (A_1, B_1; C_1; z)\nonumber\\
	& = (A_1)_n \ {}_2F_1^{(A, B; \mathbb{Y})} (A_1 +nI, B_1+nI; C_1+nI; z) \ (B_1)_n \ (C_1)_n^{-1}, \quad B_1 C_1 = C_1 B_1.
\end{align}
\end{proof}
In the similar way, we can get the differential formula for NECHMF given below in the theorem. The proof is similar to the proof for NEGHMF, so we omitted.
\begin{theorem}
	Let $A$, $B$, $\mathbb{Y}$,  $B_1$, $C_1$ and  $C_1-B_1$ be positive stable matrices in $\mathbb{C}^{r \times r}$ such that $B_1C_1 = C_1 B_1$. Then, we have the following differential formula satisfies by NECHMF
	\begin{align}
		& \frac{d^n}{dz^n} \ {}_1F_1^{(A, B; \mathbb{Y})} (B_1; C_1; z) =  {}_1F_1^{(A, B; \mathbb{Y})} (B_1+nI; C_1+nI; z)  \ (B_1)_n \ (C_1)_n^{-1}.\label{4.10}
	\end{align}
\end{theorem}
We now give the transformation formulas for NEGHMF and NECHMF.
\begin{theorem}
	Let $A$, $B$, $\mathbb{Y}$, $A_1$, $B_1$, $C_1$ and  $C_1-B_1$ be positive stable matrices in $\mathbb{C}^{r \times r}$ such that $B_1 C_1 = C_1 B_1$. Then, we have the following transformation formulas satisfied by NEGHMF
	\begin{align}
	&  {}_2F_1^{(A, B; \mathbb{Y})} (A_1, B_1; C_1; z) 
	= (1-z)^{-A_1}  \ {}_2F_1^{(A, B; \mathbb{Y})} \left(A_1, C_1-B_1; C_1; \frac{z}{z-1}\right),\label{4.11}
\end{align}
\begin{align}
	&  {}_2F_1^{(A, B; \mathbb{Y})} (A_1, B_1; C_1; z) 
	= z^{A_1}  \ {}_2F_1^{(A, B; \mathbb{Y})} \left(A_1, C_1-B_1; C_1;  1-z\right)\label{e4.11}
\end{align}
and
\begin{align}
	&  {}_2F_1^{(A, B; \mathbb{Y})} \left(A_1, B_1; C_1; \frac{z}{1+z}\right) 
	= (1+z)^{A_1}  \ {}_2F_1^{(A, B; \mathbb{Y})} \left(A_1, C_1-B_1; C_1;  - z\right).\label{a4.11}
\end{align}
\end{theorem}
\begin{proof}
Replacing $t\rightarrow 1-t$ in Equation \eqref{4.3}, we have
\begin{align}
	_2F_1^{(A, B; \mathbb{Y})} (A_1, B_1; C_1; z) & =  \int_{0}^{1} (1-z(1-t))^{-A_1} {}_1F_1 \left(A; B; -\frac{\mathbb{Y}}{t(1-t)}\right) (1-t)^{B_1-I}   \nonumber\\
	& \quad \times t^{C_1-B_1-I}   dt \ [\mathfrak{B} (B_1, C_1-B_1)]^{-1}.\label{4.12}
\end{align}
Using the matrix identity $[1-z(1-t)]^{-A_1} = (1-z)^{-A_1} \left(1 + \frac{z}{1-z} t\right)^{-A_1}$ in \eqref{4.12}, we get
\begin{align}
	_2F_1^{(A, B; \mathbb{Y})} (A_1, B_1; C_1; z) & = (1-z)^{-A_1} \int_{0}^{1} \left(1 - \frac{z}{z-1} t\right)^{-A_1} \ {}_1F_1 \left(A; B; -\frac{\mathbb{Y}}{t(1-t)}\right) (1-t)^{B_1-I}  \   \nonumber\\
	& \quad \times  t^{C_1-B_1-I} dt \ [\mathfrak{B} (B_1, C_1-B_1)]^{-1}\nonumber\\
	& = (1-z)^{-A_1}  \ {}_2F_1^{(A, B; \mathbb{Y})} \left(A_1, C_1-B_1; C_1; \frac{z}{z-1}\right).\label{4.13}
\end{align}
It completes the proof of \eqref{4.11}. To prove \eqref{e4.11} and \eqref{a4.11}, we replace $z$ by $1 - \frac{1}{z}$ and  $\frac{z}{1+z}$ in \eqref{4.11} respectively.
\end{proof}
By taking $z = 1$ and allow $A_1$ to commute with $A$, $B$, $B_1$, $C_1$ in \eqref{4.3}, we get the following relation between NEGHMF and new extended beta matrix function
\begin{align}
{}_2F_1^{(A, B; \mathbb{Y})} (A_1, B_1; C_1; 1) & =  \int_{0}^{1}  t^{B_1-I}  \ (1-t)^{C_1- A_1 -B_1-I}  \nonumber\\
& \quad \times {}_1F_1 \left(A; B; -\frac{\mathbb{Y}}{t(1-t)}\right)  dt \ [\mathfrak{B} (B_1, C_1-B_1)]^{-1}\nonumber\\
& = \mathfrak{B}_\mathbb{Y}^{A, B} (B_1, C_1- A_1 -B_1) \ [\mathfrak{B} (B_1, C_1-B_1)]^{-1}.\label{4.16}
\end{align}
Equation \eqref{4.16} allow us to write the new extension of Kummer's first theorem. We present the result in the following theorem.
\begin{theorem}
	Let $A$, $B$, $\mathbb{Y}$, $B_1$, $C_1$ and  $C_1-B_1$ be positive stable matrices in $\mathbb{C}^{r \times r}$ such that $B_1 C_1 = C_1 B_1$. Then, the new extended Kummer's first theorem is given by
	\begin{align}
&  {}_1F_1^{(A, B; \mathbb{Y})} (B_1; C_1; z) 
= \exp (z)  \ {}_1F_1^{(A, B; \mathbb{Y})} \left(C_1-B_1; C_1;  -z \right).
	\end{align} 
\end{theorem}

\section{New extended Appell and Lauricella hypergeometric matrix functions}
In this section, we introduce new extended Appell matrix  functions (NEAMFs) and new extended Lauricella matrix  function (NELMF) of three variables. More explicitly, we give new extended form of  Appell matrix  functions  $F_{1}(A_1, B_1, B_2; C_1; z, w)$, $F_{2}(A_1, B_1, B_2; C_1, C_2; z, w)$ and  Lauricella matrix  function of three variables $F^{(3)}_{D}(A_1, B_1, B_2, B_3; C_1; z, w, v)$, \cite {RD}-\cite{RD3}, in terms of the new extended beta matrix function. We also give here the integral representations for these new extended hypergeometric matrix functions. 

Let $A$, $A_1$ $B$, $B_1$, $B_2$,  $C_1$, $C_1-A_1$ and $\mathbb{Y}$ be positive stable matrices in $\mathbb{C}^{r \times r}$. 
 Then, we define new extended Appell  hypergeometric matrix  function, denoted by $F_{1}^{(A, B)}(A_1, B_1, B_2; C_1;z, w; \mathbb{Y})$, as
\begin{align}
	&F_{1}^{(A, B)}(A_1, B_1, B_2; C_1;z, w; \mathbb{Y})\notag\\
	&=  \sum_{m, n\geq 0}\mathfrak{B}^{(A, B)}_{\mathbb{Y}}(A_1+(m+n)I, C_1-A_1) \, [\mathfrak{B} (A_1, C_1-A_1)]^{-1} (B_1)_{m} (B_2)_{n}\frac{z^{m} w^{n}}{m! n!}.\label{2eq1}
\end{align}
For positive stable matrices $A$, $A'$,  $A_1$, $B$, $B'$, $B_1$, $B_2$, $C_1$, $C_2$, $C_1-B_1$, $C_2-B_2$ and $\mathbb{Y}$ in $\mathbb{C}^{r \times r}$, 
we define the new extended Appell  hypergeometric matrix  function $F_{2}^{(A, B)}(A_1, B_1, B_2; C_1, C_2; z, w; \mathbb{Y})$ as
\begin{align}
	&F_{2}^{(A, B, A', \,B')}(A_1, B_1, B_2; C_1, C_2; z, w; \mathbb{Y})\notag\\&= \sum_{m, n\geq 0}(A_1)_{m+n} \mathfrak{B}^{(A, B)}_{\mathbb{Y}}(B_1+mI, C_1-B_1) [\mathfrak{B} (B_1, C_1-B_1)]^{-1} \mathfrak{B}^{(A', B')}_{\mathbb{Y}}(B_2+nI, C_2-B_2)\nonumber\\
	& \quad \times  [\mathfrak{B} (B_2, C_2-B_2)]^{-1}\frac{z^{m}w^{n}}{m! n!}.\label{2eq2}
\end{align}
Suppose $A$, $A_1$, $B$, $B_1$, $B_2$, $B_3$, $C_1$, $C_1-A_1$ and $\mathbb{Y}$ be positive stable matrices in $\mathbb{C}^{r \times r}$. 
Then, we define the extended Lauricella  hypergeometric matrix  function $F^{(3; A, B)}_{D, \mathbb{Y}}(A_1, B_1, B_2, B_3; C_1 ; z, w, v)$ as
\begin{align}
	&F^{(3; A, B)}_{D, \mathbb{Y}}(A_1, B_1, B_2, B_3; C_1 ; z, w, v)\notag\\
	& = \sum_{m, n, p\geq 0}\mathfrak{B}^{(A, B)}_{\mathbb{Y}}(A_1+(m+n+p)I, C_1-A_1) [\mathfrak{B} (A_1, C_1-A_1)]^{-1} (B_1)_{m} (B_2)_{n} (B_3)_{p} \frac{z^{m} w^{n}v^{p}}{m! n! p!}.\label{2eq3}
\end{align}
We now turn our attention in finding the integral representations, differential formulae and  recurrence relations for new extended Appell matrix  functions (NEAMFs) and new extended Lauricella  matrix  function (NELMF) of three variables. We start with the integral representation of  $F_{1}^{(A, B)}(A_1, B_1, B_2; C_1; z, w; \mathbb{Y})$ determined in the next theorem.
\begin{theorem}\label{t1}
 Let $A$, $A_1$, $B$, $B_1$, $B_2$, $C_1$, $C_1-A_1$ and $\mathbb{Y}$ be positive stable matrices in $\mathbb{C}^{r \times r}$. 
 Then,  the  NEAMF $F_{1}^{(A, B)}(A_1, B_1, B_2; C_1; z, w; \mathbb{Y})$ can be presented in the integral form as 
	\begin{align}
		F_{1}^{(A, B)}(A_1, B_1, B_2; C_1; z, w; \mathbb{Y}) &
		=   \int_{0}^{1}  {}_1F_1 \left(A; B; -\frac{\mathbb{Y}}{u(1-u)}\right) u^{A_1-I} (1-u)^{C_1-A_1-I} \nonumber\\
		&\quad \times [\mathfrak{B} (A_1, C_1-A_1)]^{-1} (1-zu)^{-B_1} (1-wu)^{-B_2} du.\label{i1}
	\end{align}
\end{theorem}
\begin{proof} 
 Using  the  integral representation of extended beta matrix function from \eqref{3.2}  in the definition of  the NEAMF $F_{1}^{(A, B)}(A_1, B_1, B_2; C_1; z, w; \mathbb{Y})$, we get
	\begin{align}
		&F_{1}^{(A, B)}(A_1, B_1, B_2; C_1; z, w; \mathbb{Y})\notag\\ &= \sum_{m, n\geq 0}^{} \int_{0}^{1} {}_1F_1 \left(A; B; -\frac{\mathbb{Y}}{u(1-u)}\right) u^{A_1-I} (1-u)^{C_1-A_1-I}  [\mathfrak{B} (A_1, C_1-A_1)]^{-1}\notag\\ 
		&\quad \times (B_1)_{m} (B_2)_{n}\frac{{(zu)}^{m}{(wu)}^{n}}{m! n!}du.\label{a3.5} 
	\end{align}
	Applying the process discussed in \cite{RD} we can show that the sequence of matrix functions in \eqref{a3.5} is integrable and by dominated convergence theorem \cite{gf}, the summation and the integral can be interchanged in \eqref{a3.5}. Now applying the matrix identity,
	\begin{align}
		(1-z)^{-A}=\sum_{n=0}^{\infty}(A)_{n}\frac{z^{n}}{n!},\label{s11}
	\end{align}Then equation \eqref{a3.5} transform is of the form
	
	\begin{align}
		&F_{1}^{(A, B)}(A_1, B_1, B_2; C_1; z, w; \mathbb{Y})\notag\\
		& =  \int_{0}^{1} {}_1F_1 \left(A; B; -\frac{\mathbb{Y}}{u(1-u)}\right)  u^{A_1-I} (1-u)^{C_1-A_1-I} \ [\mathfrak{B} (A_1, C_1-A_1)]^{-1}  \notag\\
		&\quad \times (1-zu)^{-B_1} (1-wu)^{-B_2} du.
	\end{align}
	which gives the result.
\end{proof}
\begin{theorem}\label{ps1}
	Let $A$, $A'$, $B$, $B'$, $A_1$, $B_1$, $B_2$, $C_1$, $C_2$, $C_1-B_1$, $C_2-B_2$ and $\mathbb{Y}$ be  positive stable matrices in $\mathbb{C}^{r \times r}$. Then the NEAMF  $F_{2}^{(A, B, A', B')}(A_1, B_1, B_2; C_1, C_2; z, w; \mathbb{Y})$  defined in \eqref{2eq2} has following integral representation: 
	\begin{align}
		&F_{2}^{(A, B, A', \,B')}(A_1, B_1, B_2; C_1, C_2; z, w; \mathbb{Y})\notag\\&
		=  \int_{0}^{1}\int_{0}^{1} (1-zu-wv)^{-A_1} \, {}_1F_1 \left(A; B; -\frac{\mathbb{Y}}{u(1-u)}\right) \, u^{B_1-I} (1-u)^{C_1-B_1-I} [\mathfrak{B} (B_1, C_1-B_1)]^{-1} \notag\\
		& \quad \times {}_1F_1 \left(A'; B'; -\frac{\mathbb{Y}}{v(1-v)}\right) v^{B_2-I} (1-v)^{C_2-B_2-I} [\mathfrak{B} (B_2, C_2-B_2)]^{-1}   du \, dv.\label{s33}
	\end{align}
\end{theorem}
\begin{proof}
	Clearly  extended beta matrix function  and the NEAMF $F_{2}^{(A, B, A', B')}(A_1, B_1, B_2; C_1, C_2; z, w; \mathbb{Y})$  defined in \eqref{3.2} and \eqref{2eq2} respectively, we have
	\begin{align}
		&F_{2}^{(A, B, A', \,B')}(A_1, B_1, B_2; C_1, C_2; z, w; \mathbb{Y})\notag\\
		& = \sum_{m, n\geq 0}^{} \int_{0}^{1}\int_{0}^{1} (A_1)_{m+n} \frac{(zu)^{m} (wv)^{n}}{m! n!} {}_1F_1 \left(A; B; -\frac{\mathbb{Y}}{u(1-u)}\right) u^{B_1-I} (1-u)^{C_1-B_1-I} \notag\\
		& \quad \times [\mathfrak{B} (B_1, C_1-B_1)]^{-1} {}_1F_1 \left(A'; B'; -\frac{\mathbb{Y}}{v(1-v)}\right) v^{B_2-I} (1-v)^{C_2-B_2-I} [\mathfrak{B} (B_2, C_2-B_2)]^{-1}  du \, dv.\label{3.9}
	\end{align}
	Summation and integral in \eqref{3.9} can be interchanged by using the dominated convergence theorem. Since the summation formula \cite{SM} is
	\begin{align}
		\sum_{N\geq 0}^{}f(N) \frac{(z+w)^{N}}{N!}=\sum_{m, n\geq 0}^{}f(m+n)\frac{z^{m}}{m!}\frac{w^{n}}{n!},
	\end{align}
	we get
	\begin{align}
		&F_{2}^{(A, B,  A',\, B')}(A_1, B_1, B_2; C_1, C_2; z, w; \mathbb{Y})\notag\\&
		= \int_{0}^{1}\int_{0}^{1} \sum_{N\geq 0}^{}(A_1)_{N} \frac{(zu+wv)^{N}}{N!} {}_1F_1 \left(A; B; -\frac{\mathbb{Y}}{u(1-u)}\right) u^{B_1-I} (1-u)^{C_1-B_1-I} [\mathfrak{B} (B_1, C_1-B_1)]^{-1} \notag\\
		& \quad \times {}_1F_1 \left(A'; B'; -\frac{\mathbb{Y}}{v(1-v)}\right) v^{B_2-I} (1-v)^{C_2-B_2-I} [\mathfrak{B} (B_2, C_2-B_2)]^{-1}  du \, dv.\label{s22}
	\end{align}
	From \eqref{s11} and \eqref{s22}, we get \eqref{s33}.
\end{proof}
\begin{theorem}
	Suppose $A$, $B$, $A_1$, $B_1$, $B_2$, $B_3$, $C_1$, $C_1-A_1$ and $\mathbb{Y}$ be positive stable matrices in $\mathbb{C}^{r \times r}$. 
	Then the NELMF $F^{(3; A, B)}_{D, \mathbb{Y}}(A_1, B_1, B_2, B_3; C_1; z, w, v)$  defined in \eqref{2eq3}  have the following integral representation: 
	\begin{align}
		&F^{(3; A, B)}_{D, \mathbb{Y}}(A_1, B_1, B_2, B_3; C_1; z, w, v)\notag\\
		&=   \int_{0}^{1} {}_1F_1 \left(A; B; -\frac{\mathbb{Y}}{u(1-u)}\right) u^{A_1-I} (1-u)^{C_1-A_1-I} [\mathfrak{B} (A_1, C_1-A_1)]^{-1}  \notag\\
		& \quad \times (1-zu)^{-B_1} (1-wu)^{-B_2} (1-vu)^{-B_3} du.\label{3.12}
	\end{align}
\end{theorem}
\begin{proof}
	From \eqref{3.2} and \eqref{2eq3} together yield
	\begin{align}
		& F^{(3; A, B)}_{D, \mathbb{Y}}(A_1, B_1, B_2, B_3; C_1; z, w,v)\nonumber\\
		& =  \sum_{m, n, p\geq 0} \int_{0}^{1}  {}_1F_1 \left(A; B; -\frac{\mathbb{Y}}{u(1-u)}\right) \, u^{A_1-I} (1-u)^{C_1-A_1-I} [\mathfrak{B} (A_1, C_1-A_1)]^{-1}  \notag\\&\quad \times (B_1)_{m} (B_2)_{n}
		(B_3)_{p} \frac{(uz)^{m} (uw)^{n} (uv)^{p}}{m! n! p!}.
	\end{align}
	Now, using the matrix relation \eqref{s11} and proceeding in the similar as in Theorem~\ref{t1}, we get the required result \eqref{3.12}.
\end{proof}
\begin{theorem}
Let $A$, $A_1$, $B$, $B_1$, $B_2$, $C_1$, $C_1-A_1$ and $\mathbb{Y}$ be positive stable matrices in $\mathbb{C}^{r \times r}$ such that $A$, $A_1$, $B$, $C_1$, $\mathbb{Y}$ commutes with each other and $B_1 B_2 = B_2B_1$. Then,  we have the following differential formula satisfies by NEAMF   $F_{1}^{(A, B)}(A_1, B_1, B_2; C_1; z, w; \mathbb{Y})$: 
	\begin{align}
		& \frac{d^{m+n}}{dz^m dw^n} \left[ F_{1}^{(A, B)}(A_1, B_1, B_2; C_1; z, w; \mathbb{Y})\right]\nonumber\\
		& = (A_1)_{m+n}(C_1)^{-1}_{m+n}\notag\\&\quad\times \left[F_{1}^{(A, B)}(A_1+(m+n)I, B_1+mI, B_2+nI; C_1+(m+n)I; z, w; \mathbb{Y})\right] (B_1)_{m} (B_2)_{n}.\label{5.1}
	\end{align}
\end{theorem}
\begin{proof}
Taking the derivative of NEAMF with respect  to $z$ and $w$, we have
	\begin{align}
		&\frac{d^{2}}{dzdw} \left[ F_{1}^{(A, B)}(A_1, B_1, B_2; C_1; z, w; \mathbb{Y})\right]\notag\\
		&= \sum_{m, n\geq 0}\mathfrak{B}^{(A, B)}_{\mathbb{Y}}(A_1+(m+n)I, C_1-A_1) \, [\mathfrak{B} (A_1, C_1-A_1)]^{-1} \, (B_1)_{m} (B_2)_{n} \frac{z^{m-1} w^{n-1}}{(m-1)! (n-1)!}.\label{5.2}
	\end{align}
	Replacing $m\rightarrow m+1$, $n\rightarrow n+1$ in \eqref{5.2} and after some calculations, we get
	\begin{align}
		&\frac{d^{2}}{dzdw} \left[ F_{1}^{(A, B)}(A_1, B_1, B_2; C_1; z, w; \mathbb{Y})\right]\notag\\
		&=A_1   C^{-1}_{1}\left[F_{1}^{(A, B)}(A_1+2I, B_1+I, B_2+I; C_1+2I; z, w; \mathbb{Y})\right]B_1 B_2.\label{5.3}
	\end{align}
	Recursive application of this procedure finally gives (\ref{5.1}). 
\end{proof}
We omit the proof of the  given below two theorems.
\begin{theorem}
	For positive stable matrices $A$, $A'$,  $A_1$, $B$, $B'$, $B_1$, $B_2$, $C_1$, $C_2$, $C_1-B_1$, $C_2-B_2$ and $\mathbb{Y}$ in $\mathbb{C}^{r \times r}$ such that $A$, $A'$, $B$, $B'$, $B_1$, $B_2$, $C_1$, $C_2$  and $\mathbb{Y}$ commutes with each other. Then,  we have the following differential formula satisfies by NEAMF $F_{2}^{(A, B, A', B')}(A_1, B_1, B_2; C_1, C_2; z, w; \mathbb{Y})$:
	\begin{align}
		& \frac{d^{m+n}}{dz^m dw^n} \left[ F_{2}^{(A, B, A', B')}(A_1, B_1, B_2; C_1, C_2; z, w; \mathbb{Y})\right]\nonumber\\
		& = (A_1)_{m+n} \left[F_{2}^{(A, B, A', B')}(A_1+(m+n)I, B_1+mI, B_2+nI; C_1+mI, C_2+nI; z, w; \mathbb{Y})\right]\notag\\&\quad\times(B_1)_{m} (B_2)_{n}(C_1)^{-1}_{m+n}.\label{5.4}
	\end{align}
\end{theorem}
\begin{theorem}
	Let $A$, $A_1$, $B$, $B_1$, $B_2$, $B_3$, $C_1$, $C_1-A_1$ and $\mathbb{Y}$ be positive stable matrices in $\mathbb{C}^{r \times r}$ such that $A$, $A_1$, $B$, $C_1$, $\mathbb{Y}$ commutes with each other and $B_i B_j = B_j B_i, 1\le i, j\le 3$. Then,  we have the following differential formula satisfies by  NELMF $F^{(3; A, B)}_{D, \mathbb{Y}}(A_1, B_1, B_2, B_3; C_1; z, w, v)$:
	\begin{align}
		& \frac{d^{m+n+p}}{dz^m dw^n dv^{p}} \left[F^{(3; A, B)}_{D, \mathbb{Y}}(A_1, B_1, B_2, B_3; C_1; z, w, v)\right]\nonumber\\
		& = (A_1)_{m+n+p}(C_1)_{m+n+p} \notag\\
		& \quad \times\left[F^{(3; A, B)}_{D, \mathbb{Y}}(A_1+(m+n+p)I, B_1+mI, B_2+nI, B_3+pI; C_1+(m+n+p)I; z, w, v)\right]\notag\\&\quad\times(B_1)_{m} (B_2)_{n}(B_3)^{-1}_{p}.\label{5.5}
	\end{align}
\end{theorem}
\begin{theorem}
Let $A$, $A_1$, $B$, $B_1$, $B_2$, $C_1$, $C_1-A_1$ and $\mathbb{Y}$ be positive stable matrices in $\mathbb{C}^{r \times r}$ such that $A$, $A_1$, $B$, $C_1$, $\mathbb{Y}$ commutes with each other. Then the following recurrence relation satisfies by  $F_{1}^{(A, B)}(A_1, B_1, B_2; C_1; z, w; \mathbb{Y})$: 
	\begin{align}
&\left(B-(A+I)\right) \, F_{1}^{(A, B)}(A_1, B_1, B_2; C_1; z, w; \mathbb{Y})\notag\\&= \left(B-I\right) \, F_{1}^{(A, B-I)}(A_1, B_1, B_2; C_1; z, w; \mathbb{Y})- A \, F_{1}^{(A+I, B)}(A_1, B_1, B_2; C_1; z, w; \mathbb{Y}),\label{5.66}
\\[5pt]
&B \, F_{1}^{(A, B)}(A_1, B_1, B_2; C_1; z, w; \mathbb{Y}) - F_{1}^{(A-I, B)}(A_1, B_1, B_2; C_1; z, w; \mathbb{Y})B\notag\\
&\quad + \mathbb{Y} \, [\mathfrak{B} (A_1, C_1-A_1)]^{-1}\mathfrak{B} (A_1-I, C_1-A_1-I)\nonumber\\
& \quad \times F_{1}^{(A, B+I)}(A_1-I, B_1, B_2; C_1-2I; z, w; \mathbb{Y}).\label{5.7}
	\end{align}
	\begin{proof}
		From the following contiguous matrix relation \cite{ZM}
		\begin{align}
			\left(B-(A+I)\right)\, _{1}F_{1}(A; B; \mathbb{X})= (B-I)\,\, _{1}F_{1}(A; B-I; \mathbb{X})- A \,\, {}_{1}F_{1}(A+I; B; \mathbb{X})\end{align}
		and the integral representation of $F_{1}^{(A, B)}(A_1, B_1, B_2; C_1; z, w; \mathbb{Y})$ given in (\ref{i1}), we get (\ref{5.66}).
		
		Again, by using another contiguous matrix relation
		\begin{align}
			B \,\,_{1}F_{1}(A; B; \mathbb{X})- B \,\,_{1}F_{1}(A-I; B; \mathbb{X})= \mathbb{X} \,\,_{1}F_{1}(A; B+I; \mathbb{X}).\label{5.8}
		\end{align}
		and the integral representation of $F_{1}^{(A, B)}(A_1, B_1, B_2; C_1; z, w; \mathbb{Y})$ given in (\ref{i1}), we get (\ref{5.7}). 
	\end{proof}
\end{theorem}
We omit the proof of the  given below theorems.
\begin{theorem}
	Let $A$, $A'$, $B$, $B'$, $A_1$, $B_1$, $B_2$, $C_1$, $C_2$, $C_1-B_1$, $C_2-B_2$ and $\mathbb{Y}$ be  positive stable matrices in $\mathbb{C}^{r \times r}$ such that $A$, $A'$,  $B$, $B'$, $B_1$, $B_2$, $C_1$, $C_2$  and $\mathbb{Y}$ commutes with each other. Then the following recurrence relation satisfies by  $F_{2}^{(A, B, A', B')}(A_1, B_1, B_2; C_1, C_2; z, w; \mathbb{Y})$:
	\begin{align}
		&F_{2}^{(A, B, A', B')}(A_1, B_1, B_2; C_1, C_2; z, w; \mathbb{Y}) \left(B-(A+I)\right)\notag\\
		&= F_{2}^{(A, B-I, A', B')}(A_1, B_1, B_2; C_1, C_2; z, w; \mathbb{Y})\left(B-I\right)\nonumber\\
		& \quad - F_{2}^{(A+I, B, A', B')}(A_1, B_1, B_2; C_1, C_2; z, w; \mathbb{Y}) A.
\\[5pt]
		&F_{2}^{(A-I, B, A', B')}(A_1, B_1, B_2; C_1, C_2; z, w; \mathbb{Y})B - F_{2}^{(A, B, A', B')}(A_1, B_1, B_2; C_1, C_2; z, w; \mathbb{Y}) B \notag\\
		& = [\mathfrak{B} (B_1, C_1-B_1)]^{-1}\mathfrak{B} (B_1-I, C_1-B_1-I)\nonumber\\
		& \quad \times F_{2}^{(A, B+I, A', B')}(A_1, B_1, B_2; C_1, C_2; z, w; \mathbb{Y})\mathbb{Y}.
	\end{align}
\end{theorem}
\begin{theorem}
	Suppose $A$, $B$, $A_1$, $B_1$, $B_2$, $B_3$, $C_1$, $C_1-A_1$ and $\mathbb{Y}$ be positive stable matrices in $\mathbb{C}^{r \times r}$ such that $A$, $B$, $A_1$, $C_1$, $\mathbb{Y}$ commutes with each other. Then the  following recurrence relation satisfies by  $F^{(3; A, B)}_{D, \mathbb{Y}}(A_1, B_1, B_2, B_3; C_1; z, w, v)$: 
	\begin{align}
		&\left(B-(A+I)\right) \, F^{(3; A, B)}_{D, \mathbb{Y}}(A_1, B_1, B_2, B_3; C_1; z, w, v) \notag\\
		&= \left(B-I\right) \, F^{(3; A, B-I)}_{D, \mathbb{Y}}(A_1, B_1, B_2, B_3; C_1; z, w, v)\nonumber\\
		& \quad  - A \, F^{(3; A+I, B)}_{D, \mathbb{Y}}(A_1, B_1, B_2, B_3; C_1; z, w, v) .
\\[5pt]
		& B F^{(3; A, B)}_{D, \mathbb{Y}}(A_1, B_1, B_2, B_3; C_1; z, w, v) - BF^{(3; A-I, B)}_{D, \mathbb{Y}}(A_1, B_1, B_2, B_3; C_1; z, w, v)\notag\\
		&+[\mathfrak{B} (A_1, C_1-A_1)]^{-1}\mathfrak{B} (A_1-I, C_1-A_1-I)\nonumber\\
		& \quad \times \mathbb{Y} \, F^{(3; A, B+I)}_{D, \mathbb{Y}}(A_1-I, B_1, B_2, B_3; C_1-2I; z, w, v)=0.
	\end{align}
\end{theorem}

\end{document}